\documentclass[12pt]{amsart}       
\usepackage{txfonts}
\usepackage{amssymb}
\usepackage{eucal}
\usepackage{graphicx}
\usepackage{amsmath}
\usepackage{amscd}
\usepackage[all]{xy}           
\usepackage{amsfonts,latexsym}
\usepackage{xspace}
\usepackage{epsfig}
\usepackage{float}
\usepackage{color}
\usepackage{fancybox}
\usepackage{colordvi}
\usepackage{multicol}
\usepackage{colordvi}
\usepackage[colorlinks,final,backref=page,hyperindex,hypertex]{hyperref}
\usepackage[active]{srcltx} 
\usepackage{tikz}

\usepackage{graphicx}

\newtheorem{theorem}{Theorem}[section]

\theoremstyle{definition}
\newtheorem{defn}[theorem]{Definition}
\newtheorem{lemma}[theorem]{Lemma}

\newtheorem{prop-def}{Proposition-Definition}[section]
\newtheorem{coro-def}{Corollary-Definition}[section]

\newtheorem{remark}[theorem]{Remark}

\newtheorem{exam}[theorem]{Example}

\newcommand{\nc}{\newcommand}
\nc{\tred}[1]{\textcolor{red}{#1}}
\nc{\tblue}[1]{\textcolor{blue}{#1}}
\nc{\tgreen}[1]{\textcolor{green}{#1}}
\nc{\tpurple}[1]{\textcolor{purple}{#1}}
\nc{\btred}[1]{\textcolor{red}{\bf #1}}
\nc{\btblue}[1]{\textcolor{blue}{\bf #1}}
\nc{\btgreen}[1]{\textcolor{green}{\bf #1}}
\nc{\btpurple}[1]{\textcolor{purple}{\bf #1}}
\nc{\NN}{{\mathbb N}}
\nc{\ncsha}{{\mbox{\cyr X}^{\mathrm NC}}} \nc{\ncshao}{{\mbox{\cyr
X}^{\mathrm NC}_0}}

\newcommand{\efootnote}[1]{}

\renewcommand{\textbf}[1]{}

\newcommand{\delete}[1]{}

\nc{\mlabel}[1]{\label{#1}}  
\nc{\mcite}[1]{\cite{#1}}  
\nc{\mref}[1]{\ref{#1}}  
\nc{\mbibitem}[1]{\bibitem{#1}} 

\delete{
\nc{\mlabel}[1]{\label{#1}{\hfill \hspace{1cm}{\bf{{\ }\hfill(#1)}}}}
\nc{\mcite}[1]{\cite{#1}{{\bf{{\ }(#1)}}}}  
\nc{\mref}[1]{\ref{#1}{{\bf{{\ }(#1)}}}}  
\nc{\mbibitem}[1]{\bibitem[\bf #1]{#1}} 
}

\nc{\opa}{\ast} \nc{\opb}{\odot} \nc{\op}{\bullet} \nc{\pa}{\frakL}
\nc{\arr}{\rightarrow} \nc{\lu}[1]{(#1)} \nc{\mult}{\mrm{mult}}
\nc{\diff}{\mathfrak{Diff}}
\nc{\opc}{\sharp}\nc{\opd}{\natural}
\nc{\ope}{\circ}
\nc{\dpt}{\mathrm{d}}
\nc{\hck}{H_{RT}}
\nc{\vdf}{\calf}
\nc{\ldf}{\calf_\ell}
\nc{\hlf}{H_\ell}
\nc{\onek}{\mathbf{1}_\bfk}
\nc{\diam}{alternating\xspace}
\nc{\Diam}{Alternating\xspace}
\nc{\cdiam}{canonical alternating\xspace}
\nc{\Cdiam}{Canonical alternating\xspace}
\nc{\AW}{\mathcal{A}}

\nc{\ari}{\mathrm{ar}}

\nc{\lef}{\mathrm{lef}}

\nc{\Sh}{\mathrm{ST}}

\nc{\Cr}{\mathrm{Cr}}

\nc{\st}{{Schr\"oder tree}\xspace}
\nc{\sts}{{Schr\"oder trees}\xspace}

\nc{\vertset}{\Omega} 

\nc{\assop}{\quad \begin{picture}(5,5)(0,0)
\line(-1,1){10}
\put(-2.2,-2.2){$\bullet$}
\line(0,-1){10}\line(1,1){10}
\end{picture} \quad \smallskip}

\nc{\operator}{\begin{picture}(5,5)(0,0)
\line(0,-1){6}
\put(-2.6,-1.8){$\bullet$}
\line(0,1){9}
\end{picture}}

\nc{\idx}{\begin{picture}(6,6)(-3,-3)
\put(0,0){\line(0,1){6}}
\put(0,0){\line(0,-1){6}}
\end{picture}}

\nc{\pb}{{\mathrm{pb}}}
\nc{\Lf}{{\mathrm{Lf}}}

\nc{\lft}{{left tree}\xspace}
\nc{\lfts}{{left trees}\xspace}

\nc{\fat}{{fundamental averaging tree}\xspace}

\nc{\fats}{{fundamental averaging trees}\xspace}
\nc{\avt}{\mathrm{Avt}}

\nc{\rass}{{\mathit{RAss}}}

\nc{\aass}{{\mathit{AAss}}}

\nc{\vin}{{\mathrm Vin}}    
\nc{\lin}{{\mathrm Lin}}    
\nc{\inv}{\mathrm{I}n}
\nc{\gensp}{V} 
\nc{\genbas}{\mathcal{V}} 
\nc{\bvp}{V_P}     
\nc{\gop}{{\,\omega\,}}     

\nc{\bin}[2]{ (_{\stackrel{\scs{#1}}{\scs{#2}}})}  
\nc{\binc}[2]{ \left (\!\! \begin{array}{c} \scs{#1}\\
    \scs{#2} \end{array}\!\! \right )}  
\nc{\bincc}[2]{  \left ( {\scs{#1} \atop
    \vspace{-1cm}\scs{#2}} \right )}  
\nc{\bs}{\bar{S}} \nc{\cosum}{\sqsubset} \nc{\la}{\longrightarrow}
\nc{\rar}{\rightarrow} \nc{\dar}{\downarrow} \nc{\dprod}{**}
\nc{\dap}[1]{\downarrow \rlap{$\scriptstyle{#1}$}}
\nc{\md}{\mathrm{dth}} \nc{\uap}[1]{\uparrow
\rlap{$\scriptstyle{#1}$}} \nc{\defeq}{\stackrel{\rm def}{=}}
\nc{\disp}[1]{\displaystyle{#1}} \nc{\dotcup}{\
\displaystyle{\bigcup^\bullet}\ } \nc{\gzeta}{\bar{\zeta}}
\nc{\hcm}{\ \hat{,}\ } \nc{\hts}{\hat{\otimes}}
\nc{\barot}{{\otimes}} \nc{\free}[1]{\bar{#1}}
\nc{\uni}[1]{\tilde{#1}} \nc{\hcirc}{\hat{\circ}} \nc{\lleft}{[}
\nc{\lright}{]} \nc{\lc}{\lfloor} \nc{\rc}{\rfloor}
\nc{\curlyl}{\left \{ \begin{array}{c} {} \\ {} \end{array}
    \right .  \!\!\!\!\!\!\!}
\nc{\curlyr}{ \!\!\!\!\!\!\!
    \left . \begin{array}{c} {} \\ {} \end{array}
    \right \} }
\nc{\longmid}{\left | \begin{array}{c} {} \\ {} \end{array}
    \right . \!\!\!\!\!\!\!}
\nc{\onetree}{\bullet} \nc{\ora}[1]{\stackrel{#1}{\rar}}
\nc{\ola}[1]{\stackrel{#1}{\la}}
\nc{\ot}{\otimes} \nc{\mot}{{{\boxtimes\,}}}
\nc{\otm}{\overline{\boxtimes}} \nc{\sprod}{\bullet}
\nc{\scs}[1]{\scriptstyle{#1}} \nc{\mrm}[1]{{\rm #1}}
\nc{\margin}[1]{\marginpar{\rm #1}}   
\nc{\dirlim}{\displaystyle{\lim_{\longrightarrow}}\,}
\nc{\invlim}{\displaystyle{\lim_{\longleftarrow}}\,}
\nc{\mvp}{\vspace{0.3cm}} \nc{\tk}{^{(k)}} \nc{\tp}{^\prime}
\nc{\ttp}{^{\prime\prime}} \nc{\svp}{\vspace{2cm}}
\nc{\vp}{\vspace{8cm}} \nc{\proofbegin}{\noindent{\bf Proof: }}
\nc{\proofend}{$\blacksquare$ \vspace{0.3cm}}
\nc{\modg}[1]{\!<\!\!{#1}\!\!>}
\nc{\intg}[1]{F_C(#1)} \nc{\lmodg}{\!
<\!\!} \nc{\rmodg}{\!\!>\!}
\nc{\cpi}{\widehat{\Pi}}
\nc{\sha}{{\mbox{\cyr X}}}  
\nc{\shap}{{\mbox{\cyrs X}}} 
\nc{\shpr}{\diamond}    
\nc{\shp}{\ast} \nc{\shplus}{\shpr^+}
\nc{\shprc}{\shpr_c}    
\nc{\msh}{\ast} \nc{\zprod}{m_0} \nc{\oprod}{m_1}
\nc{\vep}{\epsilon} \nc{\labs}{\mid\!} \nc{\rabs}{\!\mid}
\nc{\sqmon}[1]{\langle #1\rangle}

\nc{\mmbox}[1]{\mbox{\ #1\ }} \nc{\dep}{\mrm{dep}} \nc{\fp}{\mrm{FP}}
\nc{\rchar}{\mrm{char}} \nc{\End}{\mrm{End}} \nc{\Fil}{\mrm{Fil}}
\nc{\Mor}{Mor\xspace} \nc{\gmzvs}{gMZV\xspace}
\nc{\gmzv}{gMZV\xspace} \nc{\mzv}{MZV\xspace}
\nc{\mzvs}{MZVs\xspace} \nc{\Hom}{\mrm{Hom}} \nc{\id}{\mrm{id}}
\nc{\im}{\mrm{im}} \nc{\incl}{\mrm{incl}} \nc{\map}{\mrm{Map}}
\nc{\mchar}{\rm char} \nc{\nz}{\rm NZ} \nc{\supp}{\mathrm Supp}

\nc{\Alg}{\mathbf{Alg}} \nc{\Bax}{\mathbf{Bax}} \nc{\bff}{\mathbf f}
\nc{\bfk}{{\bf k}} \nc{\bfone}{{\bf 1}} \nc{\bfx}{\mathbf x}
\nc{\bfy}{\mathbf y}
\nc{\base}[1]{\bfone^{\otimes ({#1}+1)}} 
\nc{\Cat}{\mathbf{Cat}}

\nc{\detail}{\marginpar{\bf More detail}
    \noindent{\bf Need more detail!}
    \svp}
\nc{\Int}{\mathbf{Int}} \nc{\Mon}{\mathbf{Mon}}
\nc{\rbtm}{{shuffle }} \nc{\rbto}{{Rota-Baxter }}
\nc{\remarks}{\noindent{\bf Remarks: }} \nc{\Rings}{\mathbf{Rings}}
\nc{\Sets}{\mathbf{Sets}} \nc{\wtot}{\widetilde{\odot}}
\nc{\wast}{\widetilde{\ast}} \nc{\bodot}{\bar{\odot}}
\nc{\bast}{\bar{\ast}} \nc{\hodot}[1]{\odot^{#1}}
\nc{\hast}[1]{\ast^{#1}} \nc{\mal}{\mathcal{O}}
\nc{\tet}{\tilde{\ast}} \nc{\teot}{\tilde{\odot}}
\nc{\oex}{\overline{x}} \nc{\oey}{\overline{y}}
\nc{\oez}{\overline{z}} \nc{\oef}{\overline{f}}
\nc{\oea}{\overline{a}} \nc{\oeb}{\overline{b}}
\nc{\weast}[1]{\widetilde{\ast}^{#1}}
\nc{\weodot}[1]{\widetilde{\odot}^{#1}} \nc{\hstar}[1]{\star^{#1}}
\nc{\lae}{\langle} \nc{\rae}{\rangle}
\nc{\lf}{\lfloor}
\nc{\rf}{\rfloor}

\nc{\QQ}{{\mathbb Q}}
\nc{\RR}{{\mathbb R}} \nc{\ZZ}{{\mathbb Z}}

\nc{\cala}{{\mathcal A}} \nc{\calb}{{\mathcal B}}
\nc{\calc}{{\mathcal C}}
\nc{\cald}{{\mathcal D}} \nc{\cale}{{\mathcal E}}
\nc{\calf}{{\mathcal F}} \nc{\calg}{{\mathcal G}}
\nc{\calh}{{\mathcal H}} \nc{\cali}{{\mathcal I}}
\nc{\call}{{\mathcal L}} \nc{\calm}{{\mathcal M}}
\nc{\caln}{{\mathcal N}} \nc{\calo}{{\mathcal O}}
\nc{\calp}{{\mathcal P}} \nc{\calr}{{\mathcal R}}
\nc{\cals}{{\mathcal S}} \nc{\calt}{{\mathcal T}}
\nc{\calu}{{\mathcal U}} \nc{\calw}{{\mathcal W}} \nc{\calk}{{\mathcal K}}
\nc{\calx}{{\mathcal X}} \nc{\CA}{\mathcal{A}}

\nc{\fraka}{{\mathfrak a}} \nc{\frakA}{{\mathfrak A}}
\nc{\frakb}{{\mathfrak b}} \nc{\frakB}{{\mathfrak B}}
\nc{\frakD}{{\mathfrak D}} \nc{\frakF}{\mathfrak{F}}
\nc{\frakf}{{\mathfrak f}} \nc{\frakg}{{\mathfrak g}}
\nc{\frakH}{{\mathfrak H}} \nc{\frakL}{{\mathfrak L}}
\nc{\frakM}{{\mathfrak M}} \nc{\bfrakM}{\overline{\frakM}}
\nc{\frakm}{{\mathfrak m}} \nc{\frakP}{{\mathfrak P}}
\nc{\frakN}{{\mathfrak N}} \nc{\frakp}{{\mathfrak p}}
\nc{\frakS}{{\mathfrak S}} \nc{\frakT}{\mathfrak{T}}
\nc{\frakX}{{\mathfrak X}}
\nc{\BS}{\mathbb{S
}}

\font\cyr=wncyr10 \font\cyrs=wncyr7
\nc{\li}[1]{\textcolor{red}{Li:#1}}
\nc{\yi}[1]{\textcolor{red}{Yi: #1}}
\nc{\xing}[1]{\textcolor{purple}{Xing:#1}}
\nc{\revise}[1]{\textcolor{red}{#1}}

\nc{\ID}{{\rm I}}\nc{\lbar}[1]{\overline{#1}}\nc{\bre}{{\rm bre}}
\nc{\sd}{\cals}\nc{\rb}{\rm RB}\nc{\A}{\rm A}\nc{\LL}{\rm L}\nc{\tx}{\tilde{X}}
\nc{\col}{\Delta_{\epsilon}}\nc{\mul}{m_{\mathrm{RT}}}\nc{\ul}{u_{RT}}\nc{\epl}{\epsilon_{RT}}
\nc{\hl}{H_{RT}}\nc{\arro}[1]{#1}\nc{\px}{P_{\tx}}\nc{\pw}{P_{\mathfrak{w}}}\nc{\pl}{B^+}
\nc{\pp}{\pl}\nc{\ppp}[1]{B^+(#1)}\nc{\dw}{\diamond_{\mathfrak{w}}}\nc{\dl}{\diamond_{\rm \ell}}
\nc{\ncshaw}{\sha^{{\rm NC}}_{\mathfrak{w}}}\nc{\ncshal}{\sha^{{\rm NC}}_{{\rm \ell}}}
\nc{\ver}{\rm V}\nc{\ld}{l}\nc{\del}{\Delta_{{\rm \ell}}}\nc{\epsl}{\epsilon_{{\rm \ell}}}
\nc{\uul}{u_{{\rm \ell}}}\nc{\oneh}{\mathbf{1}}\nc{\onew}{\mathbf{1}}
\nc{\etree}{1} \nc{\conc}{m_{RT}} \nc{\subq}{\bfk Q_l} \nc{\fid}{\unlhd}  \nc{\sfid}{\lhd}
\nc{\lhl}{\leq_{h,l}} \nc{\ghl}{\geq_{hl}}

\nc{\hrtb}{H_{RT}(X\sqcup\Omega)} \nc{\hrts}{H_{RT}(X, \Omega)}\nc{\rts}{\mathcal{T}(X, \Omega)}\nc{\rfs}{\mathcal{F}(X, \Omega)} \nc{\fm}{m_{FM}} \nc{\coll}{\Delta_\lambda}
\nc{\mn}{M_n(\bfk)}

\begin{document}

\title[Weighted infinitesimal unitary bialgebras, matrix algebras and polynomial algebras]{Weighted infinitesimal unitary bialgebras, pre-Lie, matrix algebras and polynomial algebras}
%


\author{Yi Zhang}
\address{School of Mathematics and Statistics, Lanzhou University, Lanzhou, Gansu 730000, P.\,R. China}
         \email{zhangy2016@lzu.edu.cn}

\author{Jia-Wen Zheng}
\address{School of Mathematics and Statistics, Lanzhou University, Lanzhou, Gansu 730000, P.\,R. China}
         \email{zhengjw16@lzu.edu.cn}

\author{Yan-Feng Luo} \address{School of Mathematics and Statistics,
Key Laboratory of Applied Mathematics and Complex Systems,
Lanzhou University, Lanzhou, 730000, P.\,R. China}
         \email{luoyf@lzu.edu.cn}

\date{\today}
\begin{abstract}
Motivated by comatrix coalgebras, we introduce the concept of a Newtonian comatrix coalgebra.
We construct an infinitesimal unitary bialgebra on matrix algebras,  via the construction of a suitable coproduct. As a consequence, a Newtonian comatrix coalgebra is established.
Furthermore, an infinitesimal unitary Hopf algebra, under the view of Aguiar, is constructed on matrix algebras.
By the close relationship between pre-Lie algebras and infinitesimal unitary bialgebras, we erect a pre-Lie algebra and  a new Lie algebra on matrix algebras. Finally, a weighted infinitesimal unitary bialgebra on non-commutative polynomial algebras is also given.
\end{abstract}

\subjclass[2010]{
15A30, 
16W99, 
16T10, 
17B60, 
17D25  	
}

\keywords{matrix algebra, polynomial algebra, infinitesimal bialgebra, pre-Lie algebra. }

\maketitle

\tableofcontents

\setcounter{section}{0}

\allowdisplaybreaks

\section{Introduction}

A weighted infinitesimal unitary bialgebra is a module $A$ which is simultaneously an algebra (possibly without a unit) and a coalgebra (possibly without a counit) such that the coproduct $\Delta$ is a weighted derivation of $A$ in the sense that
$$\Delta(ab)=a\cdot\Delta(b)+\Delta(a)\cdot b+\lambda (a\ot b)\,\text{ for } a, b\in A,$$
where $\lambda\in \bfk$ is a fixed constant.

Weighted infinitesimal unitary bialgebras, first appeared in~\mcite{OP10} and further studied in~\mcite{BGGZ18, GZ, ZGZ18}, are in order to give  an  algebraic meaning of  non-homogenous associative classical Yang-Baxter equations in the context of associative algebras~\mcite{OP10}.
It should be pointed out that the weighted infinitesimal unitary bialgebra is a uniform version of two infinitesimal bialgebras. The first version of infinitesimal bialgebras, also called Newtonian coalgebra~\mcite{HR92}, introduced by Joni and Rota~\mcite{JR}, are aimed at giving an algebraic framework for the calculus of Newton divided differences.
Namely, an infinitesimal bialgebra is a module $A$ which is simultaneously an algebra (possibly without a unit) and a coalgebra (possibly without a counit) such that the coproduct $\Delta$ is a derivation of $A$ in the sense that
$$\Delta(ab)=a\cdot\Delta(b)+\Delta(a)\cdot b\,\text{ for } a, b\in A.$$
The basic theory of infinitesimal bialgebras and infinitesimal Hopf algebras was developed by Aguiar~\mcite{MA, Agu01, Agu02, Aguu02}, has proven useful not only in combinatorices~\mcite{Agu02, ER98}, but in other areas of mathematics as well, such as associative Yang-Baxter equations~\mcite{MA, Aguu02}, Drinfeld's doubles~\mcite{MA, WW14} and pre-Lie algebras~\mcite{MA}.
The second version of infinitesimal bialgebras  was defined by Loday and Ronco~\mcite{LR06}
and brought new life on rooted trees by Foissy~\mcite{Foi09, Foi10} in the sense that
$$\Delta(ab)=a\cdot\Delta(b)+\Delta(a)\cdot b-a\ot b\,\text{ for } a, b\in A.$$

In 2010, Ogievetsky and Popov~\mcite{OP10} showed that weighted infinitesimal unitary bialgebras play an important role in  mathematical physics. Given a solution $r\in A\ot A$ of the non-homogenous associative classical Yang-Baxter equation,
one can construct a weighted infinitesimal unitary bialgebra~\mcite{OP10},  involving a coproduct given by
\begin{align*}
\Delta_{r}(a) :=a\cdot r- r\cdot a-\lambda (a\ot 1) \, \text{ for }\, a\in A.
\end{align*}
Weighted infinitesimal unitary bialgebras also have found applications in combinatorics. Based on Hochschild $1$-cocycle conditions in Cartier-Quillen cohomology, we can construct  weighted infinitesimal unitary bialgebras on various combinatorial objects, such as planar rooted forests and  decorated rooted forests~\mcite{Foi09, GZ, ZCGL18, ZGL182}. A  surprising phenomena showed that these weighted infinitesimal unitary bialgebras can be treated in the framework of operated algebras, via  grafting operators~\mcite{ZCGL18, ZGL182}. It has been observed that the weighted infinitesimal unitary bialgebras on rooted forests possess universal properties.
In particular, the objects studied in the well-known Connes-Kreimer Hopf algebra have a free cocycle weighted infinitesimal  bialgebraic structure~\mcite{GZ, ZCGL18, ZGL182}.
Thus it would be interesting to construct a number of weighted infinitesimal bialgebras on some combinatorial objects or various well-known algebras. This is our main goal of this paper.

Let $M_{n}(\bfk)$ be the matrix algebra and $\{E_{ij}\}_{1\leq i, j\leq n}$ the canonical $\bfk$-basis for $M_{n}(\bfk)$. Then $M_{n}(\bfk)$ has a coalgebraic structure~\mcite{Rad12} determined by
\begin{align*}
\Delta(E_{ij}):=\sum_{s=1}^{n}E_{is}\ot E_{sj}\, \text{ and }\, \varepsilon(E_{ij}):=\delta_{ij},
\end{align*}
where $\delta_{ij}$ is the Kronecker function. This comatrix coalgebra plays a foundmental role in classical bialgebras. It is almost a natural question to wonder whether we can construct an infinitesimal bialgebra on $M_{n}(\bfk)$. This paper gives a positive answer to this question.
Strongly motivated by the construction of comatrix coalgebra,
we show that $M_{n}(\bfk)$ possesses an infinitesimal unitary bialgebraic structure  with the coproduct
defined by
\begin{align*}
\col (E_{ij}):=\begin{cases}\sum_{s=i}^{j-1}E_{is}\otimes E_{(s+1)j} &\text{ if } i< j,\\
0 &\text{ if }i= j,\\
-\sum_{s=j}^{i-1}E_{is}\otimes E_{(s+1)j} &\text{ if } i> j.
\end{cases}
\end{align*}
We call this infinitesimal unitary bialgebra a Newtonian comatrix coalgebra.
We emphasize  that the Newtonian comatrix coalgebra is different from the one introduced in \mcite{ZGZ18}. See Remark~\mref{remk:Newcoal} below.
Moreover, we  equip an infinitesimal unitary bialgebra on matrix algebras with an antipode such that it is further an infinitesimal unitary Hopf algebra, under the view of Aguiar~\mcite{MA}. As a related result, an infinitesimal unitary bialgebra of weight $\lambda$ on non-commutative polynomial algebras is also given.

Pre-Lie algebras, also called Vinberg algebras, first appeared in the work of Vinberg~\mcite{Vin63} under the name left-symmetric algebras on cones and also appeared independently in the study of deformation and cohomology of associative algebras~\mcite{Ger63}. Later, there have been several interesting developments of pre-Lie algebras  in mathematices and mathematical physics, such as classical and quantum Yang-Baxter equations~\mcite{Bai07, Bor90, ES99, GS00}, Lie groups and Lie algebras~\mcite{AS05, Kim86, Med81, Vin63}, pre-Poisson algebras~\mcite{Agu00} and Poisson brackets~\mcite{BN85}, quantum field theory~\mcite{CK98} and operads~\mcite{CL01, PBG17}, $\mathcal{O}$-operators~\mcite{Bai, BGN12} and Rota-Baxter algebras~\mcite{AB08, BBGN13, Gub, LHB07}. Because of the non-associativity of pre-Lie algebras, there is not a suitable representation theory and not a complete structure theory of pre-Lie algebras~\mcite{Bai}. It is natural to consider how to construct them from some algebraic structures  which we have known. The present paper is an attempt  to construct  pre-Lie algebraic structures on some associative algebras, especially on matrix algebras.
In the algebraic framework of Aguiar~\mcite{Aguu02} for infinitesimal bialgebras, a pre-Lie algebraic structure is constructed from an arbitrary infinitesimal (unitary) bialgebra.  As an application, a pre-Lie algebraic structure and then a new Lie algebraic structure on matrix algebras are built in this paper.

{\bf Notation.}
Throughout this paper, let $\bfk$ be a unitary commutative ring unless the contrary is specified,
which will be the base ring of all modules, algebras, coalgebras, bialgebras, tensor products, as well as linear maps.
By an algebra we mean an associative algebra (possibly without a unit)
and by a coalgebra we mean a coassociative coalgebra (possibly without a counit).
For an algebra $A$, we view $A\ot A$ as an $(A,A)$-bimodule via
\begin{equation}
a\cdot(b\otimes c):=ab\otimes c\,\text{ and }\, (b\otimes c)\cdot a:= b\otimes ca,
\mlabel{eq:dota}
\end{equation}
where $a,b,c\in A$. We shall use the sign function $\mathrm{sgn}(x)$, $x\in Z$, which is given by
\begin{align*}
\mathrm{sgn}(x):=\begin{cases}
1 &\text{$\mathrm{if}$}\, x> 0,\\
0 &\text{$\mathrm{if}$}\, x=0,\\
-1 &\text{$\mathrm{if}$}\, x<0.
\end{cases}
\end{align*}

\section{Weighted infinitesimal unitary bialgebras and infinitesimal unitary Hopf algebras}\label{sec:infbi}
In this section, we first recall the concept of a weighted infinitesimal (unitary) bialgebra~\mcite{GZ}, which generalise simultaneously the one introduced by Joni and Rota~\mcite{JR} and the one initiated by Loday and Ronco~\mcite{LR06}. We then recall the notation of an infinitesimal (unitary) Hopf algebra, under the view of Aguiar.

\subsection{Weighted infinitesimal unitary bialgebras}
The following is the concept of a weighted infinitesimal (unitary) bialgebra proposed in~\mcite{GZ}.

\begin{defn}\mcite{GZ}
Let $\lambda$ be a given element of $\bfk$.
An {\bf infinitesimal bialgebra} (abbreviated {\bf $\epsilon$-bialgebra}) {\bf of weight $\lambda$} is a triple $(A,m,\Delta)$
consisting of an algebra $(A,m)$ (possibly without a unit) and a coalgebra $(A,\Delta)$ (possibly without a counit) that satisfies
\begin{equation}
\Delta (ab)=a\cdot \Delta(b)+\Delta(a) \cdot b+\lambda (a\ot b),\quad \forall a, b\in A.
\mlabel{eq:cocycle}
\end{equation}
If further $(A,m,1)$ is a unitary algebra, then the quadruple $(A,m,1, \Delta)$ is called an {\bf infinitesimal unitary bialgebra} (abbreviated {\bf $\epsilon$-unitary bialgebra}) {\bf of weight $\lambda$}.
\mlabel{def:iub}
\end{defn}

The concept of an $\epsilon$-bialgebra morphism is given as usual.

\begin{defn}\mcite{GZ}
Let $A$ and  $B$ be two $\epsilon$-bialgebras of weight $\lambda$.
A map $\phi : A\rightarrow B$ is called an {\bf infinitesimal bialgebra morphism} (abbreviated $\epsilon$-bialgebra morphism) if $\phi$ is an algebra morphism and a coalgebra morphism. The concept of an {\bf infinitesimal unitary bialgebra morphism} can be defined in the same way.
\end{defn}

\begin{remark}\label{remk:unit1}
\begin{enumerate}

\item The $\epsilon$-bialgebra introduced by Joni and Rota~\mcite{JR} is of weight zero
and the $\epsilon$-bialgebra originated from Loday and Ronco~\mcite{LR06} is of weight $-1$.

\item Let $(A,m,1, \Delta)$ be an $\epsilon$-unitary bialgebra of weight $\lambda$. Then $\Delta(1)=-\lambda(1\ot1)$ by
\begin{align*}
\Delta(1)=\Delta(1\cdot1)=1 \cdot \Delta(1) +\Delta(1)\cdot 1 +\lambda (1\ot 1)=2\Delta(1)+\lambda (1\ot 1).
\end{align*}

\item \mlabel{remk:b}
Aguiar~\mcite{MA} pointed out that there is no nonzero $\epsilon$-bialgebra of weight zero which is both unitary and counitary.
Indeed, it follows the counicity that
$$1\ot 1_{\bfk}=(\id \ot \varepsilon)\Delta(1)=0$$ and so $1=0$.

\end{enumerate}

\end{remark}

\begin{exam}\label{exam:bialgebras}
Here are some examples of weighted $\epsilon$-(unitary) bialgebras.
\begin{enumerate}
\item Any unitary algebra $(A, m,1_A)$ is an $\epsilon$-unitary bialgebra of weight zero when the coproduct is taken to be $\Delta=0$.
\item \cite[Example~2.3.5]{MA} The {\bf polynomial algebra} $\bfk \langle x_1, x_2, x_3,\ldots \rangle$ is an $\epsilon$-unitary bialgebra of weight zero with the coproduct $\Delta$
given by Eq.~(\mref{eq:cocycle}) and
 \begin{align*}
    \Delta (x_n)=\sum_{i=0}^{n-1}x_{i}\ot x_{n-1-i}=1\ot x_{n-1}+x_1\ot x_{n-2}+ \cdots +x_{n-1}\ot 1,
    \end{align*}
where we set $x_0=1$.
\item \cite[Example~2.3.2]{MA} Let $Q$ be a quiver. The {\bf path algebra} of $Q$ is the associative algebra $\bfk Q =\oplus_{n=0}^{\infty}\bfk Q_n$ whose underlying module has its basis the set of all paths $a_1a_2\cdots a_n$ of length $n\geq 0$ in $Q$. The multiplication $\ast$ of two paths $a_1a_2\cdots a_n$ and $b_1b_2\cdots b_m$ is defined by
\begin{align*}
(a_1a_2\cdots a_n)\ast(b_1b_2\cdots b_m):=\delta_{t(a_n),s(b_1)}a_1a_2\cdots a_nb_1b_2\cdots b_m,
\end{align*}
where $\delta_{t(a_n),s(b_1)}$ is the Kronecker delta. The path algebra $(\mathbf{k}Q, \ast, \col)$ is an $\epsilon$-bialgebra of weight zero with the coproduct defined by
\begin{align*}
\Delta(e):&=0 \, \text { for } e\in Q_0 \\
\Delta(a):&=s(a)\ot t(a)\,\,\, \text { for } a\in Q_1, \text{and}\\
\Delta(a_1a_2\cdots a_n):=s(a_{1})\ot a_2 \cdot\cdots a_n&+a_1\cdots a_{n-1}\ot t(a_n)+\sum_{i=1}^{n-2}a_1\cdots a_i\ot a_{i+2}\cdots a_n \text { for } n\geq 2.
\end{align*}

\item \cite[Section~2.3]{LR06}\label{exam:tensor}
Let $V$ denote a vector space. Recall that the {\bf tensor algebra} $T(V)$ over $V$ is the tensor module,
\begin{align*}
T(V)=\bfk \oplus V\oplus V^{\ot 2}\oplus \cdots \oplus V^{\ot n}\oplus \cdots,
\end{align*}
equipped with the associative multiplication $m_{T}$ called concatenation defined by
\begin{align*}
v_1\cdots v_i\ot v_{i+1}\cdots v_n \mapsto v_1\cdots v_i v_{i+1}\cdots v_n \, \text{ for }\, 0 \leq i\leq n,
\end{align*}
and with the convention that $v_1v_0=1$ and $v_{n+1}v_{n}=1$. The
tensor algebra $T(V)$ is an $\epsilon$-unitary bialgebra of weight $-1$ with the coassociative coproduct defined by
\begin{align*}
\Delta(v_1\cdots v_n):=\sum_{i=0}^{n}v_1\cdots v_i\ot v_{i+1}\cdots v_n.
\end{align*}

\item \cite[Example~2.4]{ZGL18}
The {\bf polynomial algebra} $\bfk [x]$ is an $\epsilon$-unitary bialgebra of weight $\lambda$ with the coproduct defined by
    \begin{align}\mlabel{eq:cpol}
    \Delta(1):=-\lambda (1\ot 1) \text{ and }
    \Delta(x^n):=\sum_{i=0}^{n-1}x^{i}\ot x^{n-1-i}+\lambda \sum_{i=1}^{n-1}x^{i}\ot x^{n-i} \quad \text { for } n\geq 1.
    \end{align}
When $\lambda=0$, this $\epsilon$-unitary bialgebra of weight  zero is precisely the Newtonian
coalgebra on $\bfk [x]$, which was constructed and studied by Hirschhorn and Raphael~\mcite{HR92}.

 \item \cite[Section~1.4]{Foi08} Let $(A, m, 1,\Delta, \varepsilon, c)$ be a {\bf braided bialgebra} with $A = \bfk\oplus \ker\varepsilon$ and the braiding
$c:A\ot A \to A\ot A$ given by
\begin{align*}
c:\left\{
    \begin{array}{ll}
    1\ot 1 \mapsto 1\ot 1,\\
   a\ot 1\mapsto 1\ot a, \\
    1\ot b \mapsto b\ot 1,\\
    a\ot b \mapsto 0,
    \end{array}
    \right.
\end{align*}
where $a,b\in \ker \varepsilon$ and $\lambda\in \bfk$.
Then $(A, m, 1,\Delta)$ is an $\epsilon$-unitary bialgebra of weight $-1$.

\end{enumerate}
\end{exam}

\subsection{Infinitesimal unitary Hopf algebras under the view of Aguiar}
Denote by $\Hom_{\bfk}(A,B)$ the set of linear maps from $A$ to $B$ throughout the remainder of this subsection.

\begin{defn}~\cite[Chapter~4.1]{CP94}
Let $A=(A, m, 1,\Delta, \varepsilon)$ be a classical bialgebra. Then the convolution product $\ast $ on $\Hom_{\bfk}(A,A)$ defined to be the composition:
\begin{align*}
f\ast g :=m (f\ot g)  \Delta \, \text { for }\,  f,g\in \Hom_{\bfk}(A,A),
\end{align*}
and the triple $(\Hom_{\bfk}(A,A), \ast, 1\circ \varepsilon)$ is called a convolution algebra, where $1\circ \varepsilon$ is the unit with respect to $\ast$.
The antipode $S$ is defined to be the inverse of the identity map with respect to the convolution product.
\end{defn}

\begin{remark}
The question facing us is whether we can define the antipode for an $\epsilon$-unitary bialgebra of weight zero as one does for classical bialgebras. Aguiar~\cite[Remark~2.2]{MA} answered this question ``No'' in the case of weight zero due to the lack of the unit $1\circ \varepsilon$ with respect to $\ast$, see Remark~\mref{remk:unit1}~(\mref{remk:b}).
\end{remark}
To equip the $\epsilon$-bialgebra of weight zero with an antipode,  Aguiar~\mcite{MA} introduced the notion of a circular convolution.

\begin{defn}~\cite[Section~3]{MA}
Let $(A, m, \Delta)$ be an $\epsilon$-bialgebra of weight zero. Then the {\bf circular convolution} $\circledast$ on $\Hom_{\bfk}(A,A)$ defined by
$$f\circledast g :=f\ast g+f+g \, \text {, that is, } \, (f\circledast g)(a) :=\sum_{(a)}f(a_{(1)})g({a_{(2)}})+f(a)+g(a)\, \text{ for } a\in A.$$
Note that $f\circledast 0 = f = 0\circledast f$ and so $0\in \Hom_{\bfk}(A,A)$ is the unit with respect to the circular convolution $\circledast$.
\end{defn}

Further Aguiar~\cite{MA} introduced the concept of an infinitesimal Hopf algebra via circular convolution.

\begin{defn}~\cite[Definition~3.1]{MA}
 An infinitesimal bialgebra $(A, m, \Delta)$ of weight zero is called an {\bf infinitesimal Hopf algebra} (abbreviated {\bf $\epsilon$-Hopf algebra}) if the identity map $\id\in \Hom_{\bfk}(A,A)$ is invertible with respect to the circular convolution $\circledast$. In this case, the inverse $S\in \Hom_{\bfk}(A,A)$ of $\id$ is called the {\bf antipode} of $A$. It is characterized by
\begin{equation*}
\sum_{(a)}S(a_{(1)})a_{(2)}+S(a)+a=0=\sum_{(a)}a_{(1)}S(a_{(2)})+a+S(a) \, \text{ for }\, a \in A.
\mlabel{eq:ants}
\end{equation*}
Here $\Delta(a) = \sum_{(a)} a_{(1)} \ot a_{(2)}$.
If further $(A, m, 1, \Delta)$ is a unitary algebra, then $(A, m, 1, \Delta)$ is called an {\bf infinitesimal unitary Hopf algebra}.
\mlabel{de:deha}
\end{defn}

The $\epsilon$-unitary Hopf algebra satisfies many properties analogous to those of a classical Hopf algebra~\cite[Propositions~3.7, 3.12]{MA}.

\begin{remark}
\begin{enumerate}
\item Let $(A, m, \Delta)$ be an $\epsilon$-unitary Hopf algebra with antipode $S$. Then
\begin{align*}
S(xy)=-S(x)S(y) \text{ and }\sum_{(x)} S(x_{(1)})\ot S(x_{(2)})=-\sum_{(S(x))}S(x)_{(1)} \ot S(x)_{(2)} =-\Delta S(x).
\end{align*}

\item If $(A,m,\Delta)$ is an $\epsilon$-unitary Hopf algebra with antipode $S$, then so is $(A,m^{\mathrm{op}}, \Delta^{\mathrm{cop}})$ with the same antipode $S$.

\item It follows from  Eq.~(\mref{eq:ants}) that $S(1)=-1$ by taking $a=1$.
\end{enumerate}
\end{remark}

\begin{defn}\cite[Section~4]{MA}
Let $A$ be an algebra and $C$ a coalgebra. The map $f:C\to A$ is called {\bf locally nilpotent} with respect to convolution $*$
if for each $c\in C$, there is some $n\geq 1$ such that
\begin{equation}\mlabel{eq:lnil}
f^{\ast (n)}(c) :=\sum_{(c)}f(c_{(1)})f(c_{(2)})\cdots f(c_{(n+1)})=0,
\end{equation}
where $c_{(1)}, \cdots, c_{(n+1)}$ are from the Sweedler notation $\Delta^n(c) = \sum_{(c)} c_{(1)} \ot \cdots \ot c_{(n+1)}$
and $f^{\ast n}$ is defined inductively by
\begin{align*}
f^{\ast (1)}(c) :=\sum_{(c)}f(c_{(1)})f(c_{(2)})\, \text{ and }\,
f^{\ast (n)}:=f^{\ast (n-1)}\ast f.
\end{align*}
\end{defn}

Denote by $\mathbb{R}$ and $\mathbb{C}$ the field of real numbers and the field of complex numbers, respectively.

\begin{lemma}\cite[Proposition~4.5]{MA}
Let $(A,\,m,\,\Delta)$ be an $\epsilon$-bialgebra of weight zero and $D :=m\Delta$, and let $\QQ\subseteq \bfk$. Suppose that either
\begin{enumerate}
\item
$\bfk = \mathbb{R}$ or $\mathbb{C}$ and $A$ is finite dimensional, or
\item
$D$ is locally nilpotent and char$(\bfk)=0$.
\end{enumerate}
Then $A$ is an $\epsilon$-Hopf algebra with bijective antipode $S=-\sum_{n=0}^{\infty}\frac{1}{n!}(-D)^{n}$.
\mlabel{lem:rt3}
\end{lemma}

\section{Infinitesimal unitary Hopf algebras and pre-Lie algebras on matrix algebras}\label{sec:sub}
In this section, we equip a matrix algebra $M_n(\bfk)$  with an $\epsilon$-unitary Hopf algebraic structures, in terms of a construction of the suitable coproduct. In the algebraic framework of Aguiar~\mcite{Aguu02} for $\epsilon$-(unitary) bialgebras, a pre-Lie and a new Lie algebraic structures on matrix algebras are constructed.

\subsection{An infinitesimal unitary bialgebra on matrix algebras}
In this subsection, we construct an $\epsilon$-unitary bialgebra of weight zero arising from a matrix algebra.

\begin{defn}\cite[Chapter~17]{Lam99}
Let $\bfk$ be a unitary commutative ring. A {\bf matrix algebra} $M_n(\bfk)$ is a collection of $n\times n$ matrices over $\bfk$ that form an associative algebra under matrix addition and matrix multiplication.
\end{defn}

The multiplication on $M_n(\bfk)$ will be denoted by $\frakm$. We now define a coproduct on  matrix algebra $M_n(\bfk)$  to equip it with a coalgebra structure, with an eye toward constructing an $\epsilon$-unitary bialgebra on it.

Let $E_{i j}$, $1\leq i, j\leq n$, be an  elementary matrix whose entry in the $i$-th row, $j$-th column is $1$,
and zero in all other entries. Note that $E_{ij}E_{kl}=\delta_{jk}E_{il}$. Then
 any matrix $M\in M_n(\bfk)$ can be decomposed as a linear combination of all the $n^2$  elementary matrices as follows:
 $$M=[m_{ij}]=\sum_{i, j=1}^nm_{ij}E_{ij}.$$
By linearity, we only need to define $\col(E_{ij})$ for basis elements $E_{ij}\in \mn$. Define
\begin{align}
\col (E_{ij}):=\begin{cases}\sum_{s=i}^{j-1}E_{is}\otimes E_{(s+1)j} &\text{ if } i< j,\\
0 &\text{ if }i= j,\\
-\sum_{s=j}^{i-1}E_{is}\otimes E_{(s+1)j} &\text{ if } i> j.
\end{cases}
\mlabel{eq:matrcol}
\end{align}
We observe that $M_n(\bfk)$ is closed under the coproduct $\col$.

\begin{lemma}\label{lem:comp1}
For $M, N \in M_n(\bfk)$, we have
\begin{align*}
\col(MN)=M\cdot\col(N)+\col(M)\cdot N.
\end{align*}

\end{lemma}
\begin{proof}
Let $E_{ij}$ and $E_{kl}$ be elementary matrices of $M$, $N$, respectively. Then it is enough to verify
\begin{align*}
\col(E_{ij}E_{kl})=E_{ij} \cdot \col(E_{kl})+\col(E_{ij}) \cdot E_{kl}\,  \text { for }\,  1 \leq i,j,k,l \leq n.
\end{align*}
We have two cases to consider.

\noindent{\bf Case 1.} $j\neq k$. In this case, by Eq.~(\mref{eq:matrcol}), we have
 $$\col(E_{ij}E_{kl})=\col(0)=0=E_{ij} \cdot \col(E_{kl})+\col(E_{ij}) \cdot E_{kl}$$

\noindent{\bf Case 2.} $j=k$, $i\leq l$. In this case, we only need to consider the following three subcases.

\noindent{\bf Subcase 2.1.} $i\leq j\leq l$. By Eq.~(\mref{eq:matrcol}), we have
\begin{align*}
E_{ij} \cdot \col(E_{kl})+\col(E_{ij}) \cdot E_{kl}&=E_{ij}\cdot\bigg(\sum_{s=k}^{l-1}E_{ks}\otimes E_{(s+1)l}\bigg)+\bigg(\sum_{s=i}^{j-1}E_{is}\otimes E_{(s+1)j} \bigg)\cdot E_{kl}\\
&=\sum_{s=j}^{l-1}E_{is}\otimes E_{(s+1)l}+\sum_{s=i}^{j-1}E_{is}\otimes E_{(s+1)l}
\\
&=\sum_{s=i}^{l-1}E_{is}\otimes E_{(s+1)l}=\col(E_{il})=\col(E_{ij}E_{kl}).
\end{align*}

\noindent{\bf Subcase 2.2.} $j<i\leq l$. By Eq.~(\mref{eq:matrcol}), we have
 \begin{align*}
 E_{ij} \cdot \col(E_{kl})+\col(E_{ij}) \cdot E_{kl}
 =&E_{ij}\cdot\bigg(\sum_{s=k}^{l-1}E_{ks}\otimes E_{(s+1),l}\bigg)-\bigg(\sum_{s=j}^{i-1}E_{is}\otimes E_{(s+1)j}\bigg)\cdot E_{kl}\\
 =&\sum_{s=j}^{l-1}E_{is}\otimes E_{(s+1)l}-\sum_{s=j}^{i-1}E_{is}\otimes E_{(s+1)l}\\
 =&\sum_{s=i}^{l-1}E_{is}\otimes E_{(s+1)l}=\col(E_{il})=\col(E_{ij}E_{kl})
 \end{align*}

\noindent{\bf Subcase 2.3.} $i\leq l<j$. By Eq.~(\mref{eq:matrcol}), we have
\begin{align*}
E_{ij} \cdot \col(E_{kl})+\col(E_{ij}) \cdot E_{kl}
=&-E_{ij}\cdot\sum_{s=l}^{k-1}E_{ks}\otimes E_{(s+1)l}+\sum_{s=i}^{j-1}E_{is}\otimes E_{(s+1)j}\cdot E_{kl}\\
=&-\sum_{s=l}^{j-1}E_{is}\otimes E_{(s+1)l}+\sum_{s=i}^{j-1}E_{is}\otimes E_{(s+1)l}\\
=&\sum_{s=i}^{l-1}E_{is}\otimes E_{(s+1)l}=\col(E_{il})=\col(E_{ij}E_{kl}).
 \end{align*}

\noindent{\bf Case 3.}$j=k$, $i> l$. It is similar to the proof of \noindent{\bf Case 2.}
\end{proof}

\begin{lemma}
The pair $(M_n(\bfk), \col)$ is a coalgebra (without counit). \mlabel{lem:coal}
\end{lemma}

\begin{proof}
It is enough to show the coassociative law:
\begin{align}
(\id\otimes \col)\col(E_{ij})=(\col\otimes \id)\col(E_{ij})\,\text{ for } E_{ij}\in M_n(\bfk).
\mlabel{eq:coassp}
\end{align}
When $i=j$, the Eq.~(\mref{eq:coassp}) holds trivially. By the definition of $\col$ in Eq.~(\mref{eq:matrcol}), we have two cases to consider.

\noindent{\bf Case 1.} $i<j$.
In this case, we have
\begin{align*}
(\id\otimes \col)\col(E_{ij})&=(\id\otimes \col)\left(\sum_{s=i}^{j-1}E_{is}\otimes E_{(s+1)j}\right)\quad (\text{by Eq.~(\ref{eq:matrcol})})\\
&=\sum_{s=i}^{j-1}E_{is}\otimes \col(E_{(s+1)j})=\sum_{s=i}^{j-2}E_{is}\otimes \col(E_{(s+1)j}) \quad (\text{by $\col(E_{ii})=0$ })\\
&=\sum_{s=i}^{j-2}\sum_{t=s+1}^{j-1}E_{is}\otimes E_{(s+1)t}\otimes E_{(t+1)j}=\sum_{t=i+1}^{j-1}\sum_{s=i}^{t-1}E_{is}\otimes E_{(s+1)t}\otimes E_{(t+1)j} \quad (\text{by Eq.~(\ref{eq:matrcol})})\\
&=\sum_{s=i+1}^{j-1}\sum_{t=i}^{s-1}E_{it}\otimes E_{(t+1)s}\otimes E_{(s+1)j}\quad (\text{by exchanging the index $s$ and $t$ })\\
&=\sum_{s=i+1}^{j-1}\col (E_{is})\otimes E_{(s+1)j}=\sum_{s=i}^{j-1}\col (E_{is})\otimes E_{(s+1)j}\quad (\text{by $\col(E_{ii})=0$ })\\
&=(\col\otimes \id)\left(\sum_{s=i}^{j-1}E_{is}\otimes E_{(s+1)j}\right)
=(\col\otimes \id)\col(E_{ij})\quad (\text{by Eq.~(\ref{eq:matrcol})}).
\end{align*}
\noindent{\bf Case 2.} $i>j$.
In this case, we have
\begin{align*}
(\id\otimes \col)\col(E_{ij})&=(\id\otimes \col)\left(-\sum_{s=j}^{i-1}E_{is}\otimes E_{(s+1)j}\right)\quad (\text{by Eq.~(\ref{eq:matrcol})})\\
&=-\sum_{s=j}^{i-1}E_{is}\otimes \col(E_{(s+1)j})\\
&=\sum_{s=j}^{i-1}\sum_{t=j}^{s}E_{is}\otimes E_{(s+1)t}\otimes E_{(t+1)j}
=\sum_{t=j}^{i-1}\sum_{s=t}^{i-1}E_{is}\otimes E_{(s+1)t}\otimes E_{(t+1)j} \quad (\text{by Eq.~(\ref{eq:matrcol})})\\
&=\sum_{s=j}^{i-1}\sum_{t=s}^{i-1}E_{it}\otimes E_{(t+1)s}\otimes E_{(s+1)j} \quad (\text{by exchanging the index $s$ and $t$ })\\
&=-\sum_{s=j}^{i-1}\col(E_{is})\otimes E_{(s+1)j}=-(\col\otimes \id)\left(\sum_{s=j}^{i-1}E_{is}\otimes E_{(s+1)j}\right)\\
&=(\col\otimes \id)\col(E_{ij})\quad (\text{by Eq.~(\ref{eq:matrcol})}).
\end{align*}

This completes the proof.
\end{proof}

Now we arrive at our main result in this subsection.

\begin{theorem}
The quadruple $(M_n(\bfk), \frakm, E,\col)$ is an $\epsilon$-unitary bialgebra of weight zero.
\mlabel{thm:fma}
\end{theorem}
\begin{proof}
It follows from Lemmas~\mref{lem:comp1} and ~\mref{lem:coal}.
\end{proof}

\begin{remark}\mlabel{remk:Newcoal}
\begin{enumerate}
\item Let us emphasize that this $\epsilon$-unitary bialgebra  on $M_{n}(\bfk)$ is different from the one constructed in our previous paper~\mcite{ZGZ18}. Let $L\in M_n(\bfk)$ such that $L^2=0$. Then the quadruple $(M_n(\bfk), \frakm, E,\col)$ is an $\epsilon$-unitary bialgebra of weight zero~\mcite{ZGZ18} with the coproduct
defined by
\begin{align*}
\col (M):=ML\ot L- L\ot LM \, \text{ for }\, M\in M_n(\bfk).
\end{align*}
\item In order to distinguish these two $\epsilon$-unitary bialgebras of weight zero on $M_{n}(\bfk)$, we call the $\epsilon$-unitary bialgebra of weight zero on $M_{n}(\bfk)$ with the coproduct defined by Eq.~(\mref{eq:matrcol}) a {\bf Newtonian comatrix coalgebra}.
\item We would like to emphasize that there will be a classification of $\epsilon$-unitary bialgebraic structures (weight zero) on  matrix algebras. Since $\epsilon$-unitary bialgebras of weight zero is closely related to associative Yang-Baxter equations, such a classification problem is related to the classification of  solutions of the associative Yang-Baxter equation on matrix algebras.

\end{enumerate}
 \end{remark}

\begin{exam}\mlabel{exam:matrix}
Consider the matrix algebra $M_{2}(\bfk)$. By the definition of $\col$ in Eq.~(\mref{eq:matrcol}), we have
\[\col
\begin{bmatrix}
1&0\\
1&0
\end{bmatrix}
=\col
\begin{bmatrix}
1&0\\
0&0
\end{bmatrix}
+\col
\begin{bmatrix}
0&0\\
1&0
\end{bmatrix}
=-\begin{bmatrix}
0&0\\
1&0
\end{bmatrix}\otimes
\begin{bmatrix}
0&0\\
1&0
\end{bmatrix}
\] and
\[
\col \begin{bmatrix}0&1\\0&1\end{bmatrix}
=\col \begin{bmatrix}0&1\\0&0\end{bmatrix}+\col \begin{bmatrix}0&0\\0&1\end{bmatrix}
=\begin{bmatrix}1&0\\0&0\end{bmatrix}\otimes \begin{bmatrix}0&0\\0&1\end{bmatrix}
.\]
A directly calculation shows that
\[
(\col\ot \id)\col \begin{bmatrix}1&0\\1&0\end{bmatrix}=(\id\ot \col)\col \begin{bmatrix}1&0\\1&0\end{bmatrix}=
\begin{bmatrix}0&0\\1&0\end{bmatrix}\otimes \begin{bmatrix}0&0\\1&0\end{bmatrix}\otimes \begin{bmatrix}0&0\\1&0\end{bmatrix}
\]
and
\[
(\col\ot \id)\col \begin{bmatrix}0&1\\0&1\end{bmatrix}
=(\id\ot \col)\col \begin{bmatrix}0&1\\0&1\end{bmatrix}=0.
\]
Moreover
\begin{align*}
\col\left(\begin{bmatrix}1&0\\1&0\end{bmatrix}\right)\cdot\begin{bmatrix}0&1\\0&1\end{bmatrix}+
\begin{bmatrix}1&0\\1&0\end{bmatrix} \cdot\col\left(\begin{bmatrix}0&1\\0&1\end{bmatrix}\right)&=
-\begin{bmatrix}0&0\\1&0\end{bmatrix}\otimes\begin{bmatrix}0&0\\0&1\end{bmatrix}+
\begin{bmatrix}1&0\\1&0\end{bmatrix}\otimes\begin{bmatrix}0&0\\0&1\end{bmatrix}\\
&=\begin{bmatrix}1&0\\0&0\end{bmatrix}\otimes\begin{bmatrix}0&0\\0&1\end{bmatrix}
=\col \begin{bmatrix}0&1\\0&1\end{bmatrix}\\
&=\col\left(\begin{bmatrix}1&0
\\1&0\end{bmatrix} \cdot \begin{bmatrix}0&1\\0&1\end{bmatrix}\right).
\end{align*}
\end{exam}

\subsection{An infinitesimal unitary Hopf algebra on matrix algebras}
In this subsection, we equip the $\epsilon$-unitary bialgebra  $(M_n(\bfk), \frakm, E,\col)$ of weight zero with an antipode such that it is further an $\epsilon$-unitary Hopf algebra, under the view of Aguiar~\mcite{MA}.

\begin{lemma}
Let $(M_n(\bfk), \frakm,E,\col)$ be the $\epsilon$-unitary bialgebra of weight zero in Theorem~\mref{thm:fma} and
$$D:= \frakm \col: M_n(\bfk) \to M_n(\bfk).$$
Then for each $k\geq 0$ and $M\in M_n(\bfk)$,  $D^{\ast (k+1)}(M_n(\bfk))=0$ and so $D$ is locally nilpotent.
\mlabel{lem:rt12}
\end{lemma}

\begin{proof}
It suffices to prove the first statement by induction on $k\geq 0$.
Using Sweedler notation, we may write
\begin{align*}
 \col(M)=\sum_{(M)}M_{(1)}\ot M_{(2)} \, \text { for }\, M\in M_n(\bfk).
\end{align*}
 For the initial step of $k=0$, it follows from Eq.~(\mref{eq:matrcol}) that
\begin{align*}
D^{*(1)}(M)=D^{*(1)}(M)=\sum_{(M)}D(M_{(1)})D(M_{(2)})=0,
\end{align*}
where the last step employs
\begin{align*}
D(E_{ij})= \frakm \col (E_{ij})=
\begin{cases}\sum_{s=i}^{j-1}E_{is} E_{(s+1)j}=0 &\text{ if } i< j,\\
0 &\text{ if }i= j,\\
-\sum_{s=j}^{i-1}E_{is} E_{(s+1)j}=0 &\text{ if } i> j.
\end{cases}
\end{align*}
Assume the result is true for $k=\ell$ for an $\ell\geq 1$, and consider the case when $k=\ell+1.$ Then
\begin{align*}
D^{\ast (\ell+1)}(M)=&\ (D^{\ast \ell}\ast D)(M)= \frakm (D^{\ast \ell}\ot D)\col(M)=\sum_{(M)}D^{\ast \ell}(M_{(1)})D(M_{(2)})=0.
\end{align*}
This completes the proof.
\end{proof}

\begin{theorem}
Let $\QQ\subseteq \bfk$. Then the quadruple $(M_n(\bfk),  \frakm,E, \col)$  is an $\epsilon$-unitary Hopf algebra
with the bijective antipode $S=-\sum_{n=0}^{\infty}\frac{1}{n!}(-D)^{n}$.
\mlabel{thm:rt13}
\end{theorem}

\begin{proof}
By Theorem~\mref{thm:fma}, $(M_n(\bfk), \frakm,E, \col)$ is an $\epsilon$-unitary bialgebra of weight zero.
From Lemmas~~\mref{lem:rt3} and~\mref{lem:rt12}, $(M_n(\bfk),  \frakm,\col)$ is an $\epsilon$-Hopf algebra
with bijective antipode $S=-\sum_{n=0}^{\infty}\frac{1}{n!}(-D)^{n}$.
Then the result follows from Definition~\mref{de:deha}.
\end{proof}

\subsection{A pre-Lie and a new Lie algebraic structures on matrix algebras}
In this subsection, we first recall the concept of  pre-Lie algebras and  the connection between $\epsilon$-unitary bialgebras of weight zero and pre-Lie algebras. We then   give a pre-Lie algebraic structure on a matrix algebra. Consequently, a new Lie algebraic structure on a matrix algebra is induced.

\begin{defn}~\cite{Man11}
A {\bf (left) pre-Lie algebra} is a $\bfk $-module $A$ together with a binary linear operation $\rhd: A\ot A \rightarrow A$ satisfying
\begin{align}
(a\rhd b)\rhd c-a\rhd (b\rhd c)=(b\rhd a)\rhd c-b\rhd (a\rhd c)\, \text{ for }\, a,b,c\in A.
\mlabel{eq:lpre}
\end{align}
\end{defn}

\begin{exam}
Here are two well-known pre-Lie algebras on dendriform dialgebras and Rota-Baxter algebras, respectively.
\begin{enumerate}
\item Let $(A, \prec, \succ )$ be a dendriform dialgebra. Then the multiplication $\star$ defined by $a\star b=a\prec b+a\succ b$ gives an associative algebra. In addition, define
    \begin{align*}
\rhd: A\ot A \to A, \, a\ot b \mapsto  a\succ b-a\prec b\, \text { for }\, a, b \in A
\end{align*}
Then $A$ together with $\rhd$ is a pre-Lie algebra~\mcite{Aguu02}.

\item Let $(A, P)$ be a Rota-Baxter algebra of weight $\lambda$. If the weight $\lambda=0$, then the binary operation
\begin{align*}
    \rhd: A\ot A \to A, \, a\ot b \mapsto P(a)\cdot b-b\cdot P(a)\, \text{ for }\, a, b \in A,
\end{align*}
defines a pre-Lie algebra. If the weight $\lambda=-1$, then the binary operation
\begin{align*}
    \rhd: A\ot A \to A, \, a\ot b \mapsto P(a)\cdot b-b\cdot P(a)-x\cdot y \, \text{ for } \, a, b \in A,
\end{align*}
defines a pre-Lie algebra~\mcite{AB08}.
\end{enumerate}
\end{exam}

Let $(A, \rhd)$ be a pre-Lie algebra. For any $a\in A$, let
$$L_a:A\rightarrow A, \quad b\mapsto  a \rhd b$$
be the left multiplication operator. Let
$$L:A\rightarrow \Hom_{\bfk}(A, A),\quad a\mapsto L_a.$$
The close relation between pre-Lie algebras and Lie algebras is characterized by the following two fundamental properties.

\begin{lemma}
\begin{enumerate}
\item
\cite[Theorem~1]{Ger63}
Let $(A, \rhd)$ be a pre-Lie algebra. Define for elements in $A$ a new multiplication by setting
\begin{align}
[a, b] :=a\rhd b-b\rhd a\,\text{ for }\, a,b\in A.
\mlabel{eq:lie1}
\end{align}
Then $(A, [_{-}, _{-}])$ is a Lie algebra.
\mlabel{lem:preL1}
\item \cite[Proposition~1.2]{Bai}
Eq.~(\mref{eq:lpre}) rewrites as
\begin{align*}
L_{[a,b]}=L_a\circ L_b -  L_b\circ L_a=[L_a, L_b],
\end{align*}
which implies that $L: (A, [_{-}, _{-}]) \rightarrow \Hom_{\bfk}(A, A)$ with $a\mapsto L_a$ gives a representation of the Lie algebra $(A, [_{-}, _{-}])$.
\mlabel{lem:preL2}
\end{enumerate}
\mlabel{lem:preL}
\end{lemma}

\begin{remark}
By Lemma~\mref{lem:preL}, a pre-Lie algebra induces a Lie algebra whose left multiplication operators give a representation of the associated commutator Lie algebra.
\end{remark}

Let $(A, m, 1, \Delta)$ be an $\epsilon$-unitary bialgebra of weight zero. Define
\begin{align}
\rhd: A\ot A \to A, \, a\ot b \mapsto a\rhd b:=\sum_{(b)}b_{(1)}a b_{(2)},
\mlabel{eq:preope}
\end{align}
where $b_{(1)}$ and $b_{(2)}$ are from the Sweedler notation $\Delta (b)=\sum_{(b)}b_{(1)}\ot b_{(2)}$.
The following result captures the connection from  $\epsilon$-unitary bialgebras of weigh zero to pre-Lie algebras~\mcite{Aguu02}.

\begin{lemma}~\mcite{Aguu02}
Let $(A, m, 1, \Delta)$ be an $\epsilon$-unitary bialgebra of weight zero.
Then $A$ equipped with the $\rhd$ in Eq.~(\mref{eq:preope}) is a pre-Lie algebra.
\mlabel{thm:preL}
\end{lemma}
\begin{remark}
By Aguiar's construction about the pre-Lie product,  a pre-Lie algebra from a weighted infinitesimal unitary bialgebra was derived in~\mcite{GZ}.
\end{remark}

We now  give a pre-Lie algebraic structure on matrix algebra $M_n(\bfk)$.
By linearity, we only need to define $E_{ij}\rhd_{\varepsilon}E_{kl}$ for basis elements $E_{ij}, E_{kl}\in \mn$. By Theorem~\mref{thm:fma}, $(M_n(\bfk), m, E,\col)$ is an $\epsilon$-unitary bialgebra of weight zero. Applying Theorem~\mref{thm:preL} and Eq.~(\mref{eq:preope}),  we define
\begin{align}
\rhd_{\epsilon}: M_n(\bfk) \ot M_n(\bfk) \rightarrow M_n(\bfk),\quad E_{ij}\rhd_{\epsilon} E_{kl}: =\sum_{(E_{kl})}E_{kl(1)}E_{ij} E_{kl(2)},
\mlabel{eq:preid}
\end{align}
where $E_{kl(1)}$ and $E_{kl(2)}$ are from  $\col(E_{kl})=\sum_{(E_{kl})}E_{kl(1)}\ot E_{kl(1)}$.

\begin{theorem}\mlabel{thm:preope11}
The pair $(M_n(\bfk), \rhd_{\epsilon})$ is a pre-Lie algebra and so
$(M_n(\bfk), [_{-}, _{-}]_{\epsilon})$ is a Lie algebra, where the Lie bracket given by
\begin{align}
[E_{ij}, E_{kl}]_{\epsilon}=\begin{cases}
\text{$\mathrm{sgn}$}(l-k)E_{kl} &\text{$\mathrm{if}$}\, j=i+1,l\neq k+1, (i-k+0.5)(i-l+0.5)<0,\\
\text{$\mathrm{sgn}$}(i-j)E_{ij} &\text{$\mathrm{if}$}\,j\neq i+1, l=k+1,(k-i+0.5)(k-j+0.5)< 0,\\
0& \text{$\mathrm{otherwise}$}.
\end{cases}
\mlabel{eq:lie}
\end{align}
\end{theorem}

\begin{proof}
By Theorem~\mref{thm:fma} and Lemma~\mref{thm:preL}, $(M_n(\bfk), \rhd_{\epsilon})$ is a pre-Lie algebra. The remainder follows from Lemma~\ref{lem:preL}~(\ref{lem:preL1}). Moreover, by Eqs.~(\ref{eq:matrcol}) and~(\ref{eq:preid}), we have
\begin{align*}
E_{ij}\rhd_{\epsilon} E_{kl}=\begin{cases}
E_{kl} &\text{ if } k< j=i+1\leq l,\\
-E_{kl} &\text{ if }l< j=i+1\leq k,\\
0 &\text{ otherwise }.
\end{cases}\quad
E_{kl}\rhd_{\epsilon}E_{ij}=\begin{cases}
E_{ij} &\text{ if } i< l=k+1\leq j,\\
-E_{ij} &\text{ if }j< l=k+1\leq i,\\
0 &\text{ otherwise }.
\end{cases}
\end{align*}
Applying Eq.~(\mref{eq:lie1}) in Lemma~\ref{lem:preL}~(\ref{lem:preL1}), we obtain
\begin{align}
[E_{ij}, E_{kl}]_{\epsilon}=\begin{cases}
E_{kl}-E_{ij} &\text{$\mathrm{if}$}\, j=i+1= l=k+1,\\
E_{kl} &\text{$\mathrm{if}$}\, k< j=i+1\leq l,l\neq k+1,\\
-E_{kl} &\text{$\mathrm{if}$}\, l<j=i+1\leq k,\\
-E_{ij} &\text{$\mathrm{if}$}\, i< l=k+1\leq j,j\neq i+1,\\
E_{ij} &\text{$\mathrm{if}$}\,j< l=k+1\leq i,\\
0& \text{$\mathrm{otherwise}$}.
\end{cases}
\mlabel{eq:lie2}
\end{align}
Then Eq.~(\ref{eq:lie}) follows by summing up Eq.~(\ref{eq:lie2}).
\end{proof}

\begin{remark}
\begin{enumerate}
\item We  emphasize that our Lie bracket $[_-, _-]_{\epsilon}$ is different from the classical Lie bracket $[E_{ij}, E_{kl}]=\delta_{jk}E_{il}-\delta_{li}E_{kj}$ on matrix algebras.
\item We call the Lie bracket $[_-, _-]_{\epsilon}$ an {\bf $\epsilon$-Lie bracket} which is induced by a weighted $\epsilon$-(unitary) bialgebraic structure.
\end{enumerate}
\end{remark}

\begin{exam}
Consider the matrix algebra $M_{2}(\bfk)$. Let
\begin{align*}
M=\begin{bmatrix}
1&0\\
1&0
\end{bmatrix}, \quad
N=\begin{bmatrix}
0&1\\
0&1
\end{bmatrix}.
\end{align*}
Then $M=E_{11}+E_{21}$ and $N=E_{12}+E_{22}$. By Theorem~\mref{thm:preope11}, we have
\begin{align*}
[M, N]_{\epsilon}=[E_{11}+E_{21},E_{12}+E_{22}]_{\epsilon}
=&[E_{11},E_{12}]_{\epsilon}+[E_{11},E_{22}]_{\epsilon}
+[E_{21},E_{12}]_{\epsilon}+[E_{21},E_{22}]_{\epsilon}\\
=&[E_{21},E_{12}]_{\epsilon}=E_{21}.
\end{align*}
By Example~\mref{exam:matrix} and Lemma~\mref{thm:preL}, we also have
\begin{align*}
[M, N]_{\epsilon}=\sum_{N}N_{(1)}MN_{(2)}-\sum_{M}M_{(1)}NM_{(2)}=&\begin{bmatrix}1&0\\0&0\end{bmatrix}\begin{bmatrix}
1&0\\
1&0
\end{bmatrix} \begin{bmatrix}0&0\\0&1\end{bmatrix}+\begin{bmatrix}
0&0\\
1&0
\end{bmatrix}\begin{bmatrix}
0&1\\
0&1
\end{bmatrix}
\begin{bmatrix}
0&0\\
1&0
\end{bmatrix}\\
=&\begin{bmatrix}
0&0\\
1&0
\end{bmatrix}=E_{21}.
\end{align*}
\end{exam}

\section{Weighted infinitesimal unitary bialgebras on polynomial algebras}\label{sec:prelie}

In this section, we derive a weighted $\epsilon$-unitary bialgebraic structure from a non-commutative polynomial algebra.

\begin{defn}
Let $\bfk$ be a unitary commutative ring. A  non-commutative {\bf polynomial algebra} $\bfk \langle x_1,\ldots ,x_n\rangle$ with coefficients in $\bfk$ is a free algebra generated by $\{x_1, \ldots, x_n\}$.
\end{defn}

Denote by
$$ \mathrm{Mon}:=\{x_{i_1}^{\alpha_1} x_{i_2}^{\alpha_2} \cdots x_{i_m}^{\alpha_m}\mid 1\leq i_1, i_2, \ldots, i_m\leq n, \alpha_k\in \mathbb{N}\}.$$
Then the elements in $\mathrm{Mon}$ are called {\bf monomials} in $\bfk \langle x_1,\ldots ,x_n\rangle$ which are the elements from the set of all words in $\{x_1, \ldots, x_n\}$. Note that $\mathrm{Mon}$ is a $\bfk$-basis of $\bfk \langle x_1\ldots ,x_n\rangle$ and $\mathrm{Mon}$ is a free monoid with the identity $x_0:=1$.
We denote the multiplication on $\bfk \langle x_1,\ldots ,x_n\rangle$ by $m$.

For any word $w\in \mathrm{Mon}$ with length $l(w)=n$, we define a new notation to choose some elements of $w$. Denote by
\begin{align}
w[i,j]:=
\text{the $i$-th element to the $j$-th element of $w$} &\text{ if }1\leq i\leq j\leq n.
\end{align}

\begin{exam}
Consider the polynomial algebra $\bfk\langle x,y\rangle$. Let $w=xxyxy$. Then
\begin{align*}
w[1,4]=xxyx, w[3,3]=y\, \text{ and }\, w[2,3]=xy.
\end{align*}
\end{exam}

Let us now  define a coproduct on $\bfk \langle x_1,\ldots ,x_n\rangle$ such that it is further an $\epsilon$-unitary bialgebra of weight $\lambda$. For any word $w\in \mathrm{Mon}$, define
\begin{align}
\col(w):=
\begin{cases}
0&\text{ if }w=0,\\
-\lambda (1\otimes 1)&\text{ if }w=1,\\
\sum_{i=1}^nw[1,i]\otimes w[i,n]+\lambda\sum_{i=1}^{n-1}w[1,i]\otimes w[i+1,n] &\text{ if } l(w)=n>0.
\end{cases}
\mlabel{eq:polycol2}
\end{align}
We observe that $\bfk \langle x_1,\ldots ,x_n\rangle$ is closed under the coproduct $\col$.
\begin{lemma}\mlabel{lem:code2}
Let $w_1, w_2\in \bfk \langle x_1,\ldots ,x_n\rangle$. Then
\begin{align}
\col(w_1w_2)=w_1\cdot \col (w_2)+\col(w_1)\cdot w_2+\lambda w_1\otimes w_2.
\end{align}
\end{lemma}

\begin{proof}
It suffices to consider basis elements $w_1, w_2\in \mathrm{Mon}$ by linearity. Without loss of generality, we may suppose that $l(w_1)=m\geq 0$ and $l(w_2)=n\geq 0$ and so $w_1w_2=w$ is a new word of length $m+n$.
If $w_i=0$ or $1$, then we have done. Consider $m\geq 1$ and $n\geq 1$. By Eq.~(\mref{eq:polycol2}), we have
\begin{align*}
&\ w_1\cdot \col (w_2)+\col(w_1)\cdot w_2+\lambda w_1\otimes w_2\\
=&\ w_1\cdot\left(\sum_{i=1}^{n} w_2[1,i]\otimes w_2[i,n]+\lambda\sum_{i=1}^{n-1}w_2[1,i]\otimes w_2[i+1,n]\right)\\
& +\left(\sum_{i=1}^mw_1[1,i]\otimes w_1[i,m]+\lambda\sum_{i=1}^{m-1}w_1[1,i]\otimes w_1[i+1,m]\right)\cdot w_2+\lambda w_1\otimes w_2
\quad (\text{by Eq.~(\ref{eq:polycol2})})\\
=&\ \sum_{i=m+1}^{m+n}w[1,i]\otimes w[i,m+n]+\lambda\sum_{i=m+1}^{m+n-1}w[1,i]\otimes w[i+1,m+n]\\
&+\sum_{i=1}^m w[1,i]\otimes w[i,m+n]+\lambda\sum_{i=1}^{m-1}w[1,i]\otimes w[i+1,m+n]+\lambda w_1\otimes w_2\\
=&\ \sum_{i=m+1}^{m+n}w[1,i]\otimes w[i,m+n]+\lambda\sum_{i=m+1}^{m+n-1}w[1,i]\otimes w[i+1,m+n]\\
&+\sum_{i=1}^m w[1,i]\otimes w[i,m+n]+\lambda\sum_{i=1}^{m-1}w[1,i]\otimes w[i+1,m+n]
+\lambda w[1,m]\otimes w[m+1,m+n]\\
=&\sum_{i=1}^{m+n}w[1,i]\otimes w[i,m+n]+\lambda \sum_{i=1}^{m+n-1}w[1,i]\otimes w[i+1,m+n]\\
=&\ \col(w)=\col(w_1w_2) \quad (\text{by Eq.~(\ref{eq:polycol2})},
\end{align*}
as desired.
\end{proof}

\begin{lemma}
The pair $(\bfk \langle x_1,\ldots ,x_n\rangle,\col)$ is a coalgebra (without counit).
\mlabel{lem:coalg2}
\end{lemma}
\begin{proof}
It is enough to check the coassociative law:
\begin{align}
(\col \otimes \id)\col (w)=(\id\otimes \col)\col (w)\, \text{ for }\, w\in \mathrm{Mon}.
\mlabel{eq:coaid}
\end{align}
When $w$ is $0$ or $1$, then the result holds for trivially.
Consider $l(w)\geq 1$.
On the one hand,
\begin{align*}
&(\col\otimes \id)\col(w)\\
=&(\col\otimes \id)\left(\sum_{i=1}^nw[1,i]\otimes w[i,n]+\lambda\sum_{i=1}^{n-1}w[1,i]\otimes w[i+1,n]\right)
\quad (\text{by Eq.~(\ref{eq:polycol2})})\\
=&\sum_{i=1}^n\col(w[1,i])\otimes w[i,n]+\lambda\sum_{i=1}^{n-1}\col(w[1,i])\otimes w[i+1,n]\\
=&\sum_{i=1}^n\left(\sum_{j=1}^{i}w[1,j]\otimes w[j,i]+\lambda\sum_{j=1}^{i-1}w[1,j]\otimes w[j+1,i]\right)\otimes w[i,n]\\
&+\lambda\sum_{i=1}^{n-1}\left(\sum_{j=1}^{i}w[1,j]\otimes w[j,i]+\lambda\sum_{j=1}^{i-1} w[1,j]\otimes w[j+1,i]\right)\otimes w[i+1,n]\quad (\text{by Eq.~(\ref{eq:polycol2})})\\
=&\sum_{i=1}^n\sum_{j=1}^{i}w[1,j]\otimes w[j,i]\otimes w[i,n]+\lambda\sum_{i=1}^n\sum_{j=1}^{i-1}w[1,j]\otimes w[j+1,i]\otimes w[i,n])\\
&+\lambda\sum_{i=1}^{n-1}\sum_{j=1}^{i}w[1,j]\otimes w[j,i]\otimes w[i+1,n]+\lambda^2\sum_{i=1}^{n-1}\sum_{j=1}^{i-1} w[1,j]\otimes w[j+1,i]\otimes w[i+1,n]\\
=&\sum_{i=1}^n\sum_{j=1}^{i}w[1,j]\otimes w[j,i]\otimes w[i,n]+\lambda\sum_{i=2}^n\sum_{j=1}^{i-1}w[1,j]\otimes w[j+1,i]\otimes w[i,n])\\
&+\lambda\sum_{i=1}^{n-1}\sum_{j=1}^{i}w[1,j]\otimes w[j,i]\otimes w[i+1,n]+\lambda^2\sum_{i=2}^{n-1}\sum_{j=1}^{i-1} w[1,j]\otimes w[j+1,i]\otimes w[i+1,n].
\end{align*}
On the other hand,
\begin{align*}
&(\id\otimes \col)\col(w)\\
=&(\id\otimes \col)\left(\sum_{i=1}^nw[1,i]\otimes w[i,n]+\lambda\sum_{i=1}^{n-1}w[1,i]\otimes w[i+1,n]\right)
\quad (\text{by Eq.~(\ref{eq:polycol2})})\\
=&\sum_{i=1}^nw[1,i]\otimes \col(w[i,n])+\lambda\sum_{i=1}^{n-1}w[1,i]\otimes \col(w[i+1,n])\\
=&\sum_{i=1}^nw[1,i]\otimes\left(\sum_{j=i}^nw[i,j]\otimes w[j,n]+\lambda\sum_{j=i}^{n-1}w[i,j]\otimes w[j+1,n]\right)\\
&+\lambda\sum_{i=1}^{n-1}w[1,i]\otimes\left(\sum_{j=i+1}^n w[i+1,j]\otimes w[j,n]+\lambda\sum_{j=i+1}^{n-1}w[i+1,j]\otimes w[j+1,n]\right)\quad (\text{by Eq.~(\ref{eq:polycol2})})\\
=&\sum_{i=1}^n\sum_{j=i}^nw[1,i]\otimes w[i,j]\otimes w[j,n]+\lambda\sum_{i=1}^n\sum_{j=i}^{n-1}w[1,i]\otimes w[i,j]\otimes w[j+1,n]\\
&+\lambda\sum_{i=1}^{n-1}\sum_{j=i+1}^nw[1,i]\otimes w[i+1,j]\otimes w[j,n]+\lambda^2\sum_{i=1}^{n-1}\sum_{j=i+1}^{n-1}w[1,i]\otimes w[i+1,j]\otimes w[j+1,n]\\
=&\sum_{i=1}^n\sum_{j=i}^nw[1,i]\otimes w[i,j]\otimes w[j,n]+\lambda\sum_{i=1}^{n-1}\sum_{j=i}^{n-1}w[1,i]\otimes w[i,j]\otimes w[j+1,n]\\
&+\lambda\sum_{i=1}^{n-1}\sum_{j=i+1}^nw[1,i]\otimes w[i+1,j]\otimes w[j,n]+\lambda^2\sum_{i=1}^{n-2}\sum_{j=i+1}^{n-1}w[1,i]\otimes w[i+1,j]\otimes w[j+1,n]\\
=&\sum_{j=1}^n\sum_{i=1}^jw[1,i]\otimes w[i,j]\otimes w[j,n]+\lambda\sum_{j=1}^{n-1}\sum_{i=1}^jw[1,i]\otimes w[i,j]\otimes w[j+1,n]\\
&+\lambda\sum_{j=2}^n\sum_{i=1}^{j-1}w[1,i]\otimes w[i+1,j]\otimes w[j,n]+\lambda^2\sum_{j=2}^{n-1}\sum_{i=1}^{j-1}w[1,i]\otimes w[i+1,j]\otimes w[j+1,n]\\
& \hspace{7cm}(\text{by exchanging the order of $i$ and $j$ in all sums})\\
=&\sum_{i=1}^n\sum_{j=1}^iw[1,j]\otimes w[j,i]\otimes w[i,n]+\lambda\sum_{i=1}^{n-1}\sum_{j=1}^iw[1,j]\otimes w[j,i]\otimes w[i+1,n]\\
&+\lambda\sum_{i=2}^n\sum_{j=1}^{i-1}w[1,j]\otimes w[j+1,i]\otimes w[i,n]+\lambda^2\sum_{i=2}^{n-1}\sum_{j=1}^{i-1}w[1,j]\otimes w[j+1,i]\otimes w[i+1,n]\\
& \hspace{7cm}(\text{by exchanging the index of $i$ and $j$}).
\end{align*}
This completes the proof.
\end{proof}
Now we state our main result in this section.
\begin{theorem}
The quadruple $(\bfk \langle x_1,\ldots ,x_n\rangle, m, 1,\col)$ is an $\epsilon$-unitary bialgebra of weight $\lambda$.
\mlabel{thm:fma1}
\end{theorem}
\begin{proof}
It follows from the Lemmas~\mref{lem:code2} and~\mref{lem:coalg2}.
\end{proof}

\begin{exam}\mlabel{exam:poly}
Consider the polynomial algebra $\bfk\langle x,y\rangle$. Let $w_1=xy$, $w_2=yxy$ be two words in $\bfk\langle x,y\rangle$. By the definition of $\col$ in Eq.~(\mref{eq:polycol2}), we have
\begin{align}
\col(w_1)&=xy\otimes y+x\otimes xy+\lambda x\otimes y
\mlabel{eq:e1}
\end{align}
and
\begin{align}
\col(w_2)&=yxy\otimes y+yx\otimes xy+y\otimes yxy+\lambda(yx\otimes y +y\otimes xy).
\mlabel{eq:e2}
\end{align}
Then
\begin{align*}
&\ (\col\otimes\id)\col(w_1)\\
=&\ (\col\otimes\id)(xy\otimes y +x\otimes xy+\lambda x\otimes y)\quad
(\text{by Eq.~(\ref{eq:e1})})\\
=&\ (xy\otimes y+x\otimes xy+\lambda x\otimes y)\otimes y+x\otimes x\otimes xy+\lambda x\otimes x\otimes y\quad
(\text{by Eq.~(\ref{eq:polycol2})})\\
=&\ x\otimes xy\otimes y+xy\otimes y\otimes y+x\otimes x\otimes xy+\lambda x\otimes y\otimes y+\lambda x\otimes x\otimes y\\
=&\ x\otimes(xy\otimes y+x\otimes xy+\lambda x\otimes y)+xy\otimes y\otimes y+\lambda x\otimes y\otimes y\\
=&\ x\otimes\col(xy)+xy\otimes\col(y)+\lambda x\otimes\col(y)\quad (\text{by Eq.~(\ref{eq:polycol2})})\\
=&\ (\id\otimes\col)(x\otimes xy+xy\otimes y+\lambda x\otimes y)\\
=&\ (\id\otimes\col)\col(xy)\quad (\text{by Eq.~(\ref{eq:e1})}).
\end{align*}
Similarly,
\begin{align*}
&\ (\col\otimes\id)\col(w_2)\\
=&\ (\col\otimes\id)\Big(yxy\otimes y+yx\otimes xy+y\otimes yxy+\lambda(yx\otimes y +y\otimes xy)\Big)\quad (\text{by Eq.~(\ref{eq:e2})})\\
=&\ y\otimes\Big(yxy\otimes y +yx\otimes xy+y\otimes yxy+\lambda(yx\otimes y +y\otimes xy)\Big)
+yx\otimes(xy\otimes y+x\otimes xy+\lambda x\otimes y)\\
&+yxy\otimes y\otimes y+\lambda\Big(y\otimes(x\otimes xy+xy\otimes y+\lambda x\otimes y)\Big)+\lambda yx\otimes y\otimes y\\
=&\ y\otimes\col(yxy)+yx\otimes\col(xy)+yxy\otimes\col(y)+\lambda y\otimes\col(xy)+\lambda yx\otimes\col(y)\quad (\text{by Eq.~(\ref{eq:polycol2})})\\
=&\ (\id\otimes\col)(y\otimes yxy+yx\otimes xy+yxy\otimes y+\lambda y\otimes xy+\lambda yx\otimes y)\\
=&\ (\id\otimes\col)\col(w_2)\quad (\text{by Eq.~(\ref{eq:e2})}).
\end{align*}
A directly calculation shows that
\begin{align*}
 w_1\cdot\col(w_2)+\col(w_1)\cdot w_2+\lambda w_1\otimes w_2
=&\col(xy)\cdot yxy+xy\cdot\col(yxy)+\lambda xy\otimes yxy\\
=&x\otimes xyyxy+xy\otimes yyxy+xyy\otimes yxy+xyyx\otimes xy+xyyxy\otimes y\\
 &+\lambda(x\otimes yyxy+xy\otimes yxy+xyy\otimes xy+xyyx\otimes y)\\
=&\col(xyyxy)=\col(w_1w_2)\quad (\text{by Eq.~(\ref{eq:polycol2})}).
\end{align*}
\end{exam}

\smallskip

\noindent {\bf Acknowledgments}:
We thank the anonymous referee for valuable suggestions helping to improve the paper.
\medskip

\noindent {\bf Funds}:
This work was supported by the National Natural Science Foundation
of China (Grant No.\@ 11771191).

\end{document}